\documentclass{elsarticle}

\usepackage{verbatim}
\usepackage{amssymb}
\usepackage{amsmath}
\newtheorem{thm}{Theorem}
\newtheorem{lem}[thm]{Lemma}
\newtheorem{prop}[thm]{Proposition}
\newtheorem{dfn}[thm]{Definition}
\newtheorem{rem}[thm]{Remark}

\newenvironment{proof}[1][Proof]{\textbf{#1.} }{\ \rule{0.5em}{0.5em}}

\let\eps=\varepsilon
\let\kappa=\varkappa
\def\bt{\hbox{$\bf \cdot$}}
\def\osc{\hbox{\rm Osc}}

\let\eps=\varepsilon
\let\a=\alpha
\let\b=\beta

\let\d=\delta
\let\D=\Delta
\let\s=\sigma

\let\l=\lambda

\let\ls=\leqslant
\let\gs=\geqslant

\let\cal=\mathcal

\def\pr{\hbox{\bf P}}
\def\E{\hbox{\bf E}}

\title{The rate of convergence  of estimate for Hurst index  of fractional Brownian motion involved into stochastic differential equation\tnoteref{t1}}
\tnotetext[t1]{This paper was supported by Research program
at International Center of Mathematics Meetings (CIRM,  http://www.cirm.univ-mrs.fr/)}

\author[kk]{K. Kubilius\corref{cor1}}
\ead{kestutis.kubilius@mii.vu.lt}
\author[myus]{Y. Mishura}
\ead{myus@univ.kiev.ua}
\cortext[cor1]{Corresponding author}
\address[kk]{Vilniaus university Institute of Mathematics and Informatics,
                              Akademijos 4, LT-08663 Vilnius, Lithuania}
\address[myus]{National Taras Shevchenko Kyiv University,
                              Volodymyrska 64, 01601 Kiev, Ukraine}

\date{}

\begin{document}
\begin{abstract}
We consider stochastic differential equation involving pathwise integral with respect to fractional Brownian motion. The estimates for the Hurst parameter are constructed according to first- and second-order quadratic variations of observed values of the solution. The rate of convergence of these estimates to the true value of a parameter is established.
\end{abstract}

\maketitle

\section{Introduction}

Consider  stochastic differential equation
\begin{equation}\label{SIE}
X_t=\xi + \int_0^t f(X_s)\,d s + \int_0^t g(X_s)\,d B^H_s,\quad t\in[0,T],\ T>0,
\end{equation}
where $f$ and $g$ are measurable functions, $B^H$ is a fractional Brownian motion (fBm) with Hurst index  $1/2<H<1$, $\xi$ is a random variable.  It is well-known that almost all sample paths of $B^H$  have bounded
$p$-variations for $p>1/H.$ Therefore it is natural to define  the integral with respect to fractional Brownian motion as pathwise Riemann-Stieltjes integral  (see,  e.g.,  \cite{y} for the original definition and \cite{dn1} for the advanced results).

A solution of stochastic differential equation \eqref{SIE}  on a given filtered
probability space  $(\Omega,{\mathcal F},{\bf P}, \mathbb{F}=\{\mathcal F_t\}, t\in[0,T])$, with respect to the
fixed fBm $(B^H,\mathbb{F}),$ $1/2<H<1$ and with $\mathcal F_0$-measurable initial condition $\xi$  is an adapted to
the filtration $\mathbb{F}$ continuous process $X=\{X_t\colon\ 0\ls  t\ls  T\}$ such
that $X_0=\xi$ a.s.,
\[
{\bf P}\bigg(\int_0^t\vert f(X_s)\vert\,ds+\big\vert\int_0^t
g(X_s)\,dB^H_s\big\vert<\infty\bigg)=1\qquad\mbox{for every}\ 0\ls  t\ls  T,
\]
and its almost all sample paths satisfy (\ref{SIE}).

For $0<\alpha\ls 1$, $\mathcal{C}^{1+\a}(\mathbb{R})$ denotes the set of all
$\mathcal{C}^{1}$-functions $g$: $\mathbb{R}\to\mathbb{R}$ such that
\[
\sup_{x}\vert g^\prime(x) \vert + \sup_{x\neq y}\frac{\vert g^\prime(x) - g^\prime(y) \vert}{\vert x
- y \vert^\a}<\infty.
\]

Let $f$ be a Lipschitz function and let
$g\in\mathcal{C}^{1+\a}(\mathbb{R})$, $\frac1H-1<\alpha\ls 1$.    Then there
exists a unique solution of equation (\ref{SIE}) with almost all sample paths in
the class of all continuous functions defined on $[0,T]$ with bounded $p$-variation for any $p>\frac1H$
(see \cite{du},
\cite{ly}, \cite{ly1} and  \cite{kk}). Different (but similar in many features) approach to the integration with respect to fractional Brownian motion based on the integration in Besov spaces and corresponding stochastic differential equations were studied in \cite{nr}, see also \cite{biagini} and \cite{mishura} where the different approaches to  stochastic integration and to stochastic differential equations  involving fractional Brownian motion  are summarized.

The main goal of the present paper is to establish the rate of convergence of two estimates of Hurst parameter  to the true value of a parameter. The estimates are based on the two types of the quadratic variations of the observed solution to  stochastic differential equation  involving  the integral with respect to fractional Brownian motion and considered on the fixed interval $[0,T]$. The paper is organized as follows: Section 2 contains some preliminary information. More precisely, subsection 2.1  describes  the properties of $p$-variations and of the  integrals with respect to the functions of bounded $p$-variations while subsection 2.2 contains the results on the asymptotic behavior of the normalized first- and second-order quadratic variations of fractional Brownian motion. Section 3 describes the rate of convergence of the first- and second-order quadratic variations of the solution to stochastic differential equation involving fBm. Section 4 contains the main result concerning  the rate of convergence of the constructed estimates of Hurst index to its true value when  the diameter of partitions of the interval $[0,T]$ tends to zero. Section 5 contains simulation results.

\section{Preliminaries}

\subsection{The functions of bounded p-variation}

First, we mention some information concerning $p$-variation and the functions of bounded $p$-variation. It is  containing, e.g., in  \cite{dn1} and \cite{y}.
Let interval $[a,b]\subset \mathbb{R}.$ Consider the following class of functions:
\[
\mathcal{W}_p\big([a,b]\big):=\big\{f: [a,b]\to\mathbb{R}{:}\ v_p\big(f;
[a,b]\big)<\infty\big\},
\]
where
\[
v_p\big(f; [a,b]\big)=\sup_{\pi}\sum_{k=1}^n \big\vert
f(x_k)-f(x_{k-1})\big\vert^p.
\]
 Here $\pi=\{x_i{:}\ i=0,\ldots,n\}$ stands for any finite partition of $[a,b]$ such that
$a=x_0<x_i<\cdots<x_n=b$. Denote $\Pi([a,b])$ the class of such partitions. We say that function $f$ has
 bounded $p$-variation on $[a;b]$ if $v_p(f;[a,b])< \infty$.

Let $V_p(f):=V_p(f;[a,b])=v_p^{1/p}(f; [a,b])$. Then for any fixed $f$ we have that $V_p(f)$ is a non-increasing
function of $p$. It means that for any $0<q<p$ the relation $V_p(f)\ls V_q(f)$ holds.

Let $a<c<b$  and let $f\in{\mathcal W}_p([a,b])$ for some $p\in(0,\infty).$ Then
\begin{gather*}
v_p\big(f;[a,c]\big)+v_p\big(f;[c,b]\big)\ls
v_p\big(f;[a,b]\big),\\[4pt]
V_p\big(f;[a,b]\big)\ls V_p\big(f;[a,c]\big)+V_p\big(f;[c,b]\big).
\end{gather*}

Let $f\in {\cal  W}_q([a,b])$ and $h\in {\cal W}_p([a,b])$, where $p^{-1}+q^{-1}>1$. Then the well-known Love-Young inequality states that
\begin{equation}\label{1.2}
\Bigg\vert\int_a^bf\,d h-f(y)\big[ h(b)-h(a) \big]\Bigg\vert \ls C_{p,q}
V_q\big(f;[a,b]\big)V_p\big(h;[a,b]\big),
\end{equation}
whence
\begin{equation}\label{1.3}
V_p\Bigg(\int_a^{\bt} f\,d h;[a,b]\Bigg)\ls C_{p,q}
V_{q,\infty}\big(f;[a,b]\big) V_p\big(h;[a,b]\big).
\end{equation}
Here $V_{q,\infty}(f;[a,b])=V_q(f;[a,b])+\sup_{a\ls x\ls b}\vert f(x)\vert$,
$C_{p,q}=\zeta(p^{-1}+q^{-1})$ and $\zeta(s)=\sum_{n\gs 1} n^{-s}$ is the Riemann zeta function.
Further, for any $y\in[a,b]$
\begin{align}\label{1.3a}
V_p\Bigg(\int_a^{\bt} [f(x)-f(y)]\,d h(x);[a,b]\Bigg)\ls& C_{p,q}
\big[V_q\big(f;[a,b]\big)\nonumber\\+\sup_{a\ls x\ls b}\vert f(x)-f(y)\vert\big] V_p\big(h;[a,b]\big)
\ls& 2C_{p,q}V_q\big(f;[a,b]\big) V_p\big(h;[a,b]\big).
\end{align}
Denote
$$|A|_\infty=\sup_{x\in \mathbb{R}}|A(x)|,\;\;|A|_\alpha=\sup_{x,y\in \mathbb{R}}\frac{|A(x)-A(y)|}{|x-y|^\alpha}.
$$
Let $F$ be a Lipschitz function and let $G\in\mathcal{C}^{1+\a}(\mathbb{R})$ with $0<\alpha\ls 1$ and $1\ls p< 1+\a$.
Then  for any $h\in\mathcal{W}_p([a,b])$
\begin{align}
V_{p,\infty}\big( F(h);[a,b]\big)\ls& L V_p\big( h;[a,b]\big) +\sup_{a\ls x\ls b}\vert F(h(x))-F(h(a))\vert+\vert F(h(a))\vert\nonumber\\
\ls& 2L  V_p\big( h;[a,b]\big)+\vert F(h(a))\vert \label{1.3b},\\
V_{p/\alpha,\infty}\big( G(h);[a,b]\big)\ls& V_{p,\infty}\big( G(h);[a,b]\big)\ls  2\vert G^\prime\vert_\infty  V_p\big( h;[a,b]\big)+\vert G(h(a))\vert\label{1.3c},
\end{align}
and
\begin{align}
&V_{p/\alpha,\infty}\big( G^\prime(h);[a,b]\big)\nonumber\\
&\quad\ls\vert G^\prime\vert_\alpha V_p^\alpha\big(h;[a,b]\big) +\sup_{a\ls x\ls b}\vert G^\prime(h(x))-G^\prime(h(a))\vert+\vert G^\prime(h(a))\vert\nonumber\\
&\quad\ls 2\vert G^\prime\vert_\alpha  V^\alpha_p\big( h;[a,b]\big)+\vert G^\prime(h(a))\vert\,\label{1.3d}.
\end{align}

Let $f \in\mathcal{W}_p([a,b])$ and $p_1>p>0.$ Then
\begin{equation}  \label{1.4}
V_{p_1}\bigl( f;[a,b] \bigr)\ls \hbox{\rm Osc}(f;[a,b])^{(p_1-p)/p_1}
V_p^{p/p_1}\bigl( f;[a,b] \bigr),
\end{equation}
where $\hbox{\rm Osc}(f;[a,b])=\sup\{\vert f(x)-f(y)\vert \colon\ x,y\in [a,b]\}$.

Take functions $f_1,f_2\in {\cal  W}_p([a,b])$, $0<p<\infty$. Then $f_1f_2\in {\cal  W}_p([a,b])$ and
\begin{equation}\label{1.4a}
V_{p}\big(f_1f_2;[a,b]\big)\leq V_{p,\infty}\big(f_1f_2;[a,b]\big)\ls C_p V_{p,\infty}\big(f_1;[a,b]\big)
V_{p,\infty}\big(f_2;[a,b]\big).
\end{equation}
Let $f_1 \in{\mathcal W}_q([a,b])$ and $f_2\in{\mathcal W}_p([a,b]).$ Then it follows from Young's version of
H\"older's inequality that for any partition
$\pi\in\Pi([a,b])$ and for any $p^{-1}+q^{-1}\gs 1$
\begin{equation}\label{1.5}
\sum_i V_{q}\big( f_1;[x_{i-1},x_i] \big)V_{p}\big( f_2;[x_{i-1},x_i] \big) \ls V_q\big(
f_1;[a,b] \big)V_p\big( f_2;[a,b] \big).
\end{equation}

Second, we state some facts from the theory of Riemann-Stieltjes integration. Let $f\in {\cal  W}_q([a,b])$ and $h\in {\cal  W}_p([a,b])$ with $0<p<\infty$, $q>0,$ $1/p+\allowbreak 1/q>1.$ Let symbol (R) stands for the Riemann integration, and (RS) stands  for Riemann-Stieltjes integration.  Then
 integral $(RS)\int_a^b f\,\mathrm{d}h$ exists under the additional assumption that $f$ and $h$ have no common discontinuities.

\begin{prop}\label{2.1} Let $f: [a, b] \to \mathbb{R}$ be such function that for some $1 \ls p < 2$ $f\in \mathcal{CW}_p([a,b])$. Also, let $F: \mathbb{R} \to \mathbb{R}$ be a differentiable function with locally Lipschitz derivative $F^\prime$. Then  composition $F^\prime(f)$ is Riemann-Stieltjes integrable with respect to $f$ and
\[
F(f(b))-F(f(a))=(RS)\int_a^b F^\prime(f(x))\,df(x).
\]
\end{prop}
Furthermore, the following substitution rule holds.
\begin{prop}\label{2.2} Let $f_1,f_2$ and $f_3$ be functions from $\mathcal{CW}_p([a, b])$,
$1 \ls p < 2$. Then
\[
(RS)\int_a^b f_1(x)\,d\bigg((RS)\int_a^x f_2(y)\,df_3(y)\bigg)=(RS)\int_a^b f_1(x)f_2(x)\,df_3(x).
\]
\end{prop}
Finally, assume that
\[
F_1(x)=(R)\int_a^x f_1(y)\,dy\qquad\mbox{and}\qquad F_2(x)=(RS)\int_a^x f_2(y)\,df_3(y),
\]
where  $f_1$ is continuous function, $f_2, f_3\in\mathcal{CW}_p([a,b])$ for some $1 \ls p < 2$, and $Q$  is a  differentiable function with locally Lipschitz derivative $q$. It follows  from  Propositions \ref{2.1} and \ref{2.2} that
\begin{align}\label{2.3}
&Q(F_1(x)+F_2(x))-Q(0)=\int_a^x q(F_1(y)+F_2(y))\,d(F_1(y)+F_2(y))\nonumber\\
&\quad=\int_a^x q(F_1(y)+F_2(y))f_1(y)\,dy+\int_a^x q(F_1(y)+F_2(y))f_2(y)\,d f_3(y).
\end{align}

\subsection{Asymptotic property of the first- and second-order quadratic variations of fractional Brownian motion}

Consider the fractional Brownian motion (fBm) $B^H=\{B^H_t, t\in[0,T]\}$ with Hurst index $H\in(\frac12,1)$.
Its sample paths  are almost all  locally H\" older
 up to order $H$. Moreover, for any $0<\gamma<H$ we have that $L^{H,\gamma}_T:= \sup_{\substack{s\ne t\\s,t\ls T}}{\frac{\vert  B^H_t-B^H_s\vert}{\vert
t-s\vert^\gamma}}$ is finite a.s. and even more, ${\bf E}\big(L^{H,\gamma}_T\big)^k<\infty$
for any $ k\gs 1.$ The following estimate for the $p$-variation of fBm is evident:
\begin{equation}\label{1.6}
V_p\big(B^H;[s,t]\big)\ls L^{H,1/p}_T\, (t-s)^{1/p}, \vspace{1pt}
\end{equation}
where $s<t\ls T,$ $p>1/H$.

Let $\pi_n=\{0=t^n_0<t^n_1<\cdots<t^n_n=T\}$, $T>0$, be a sequence of
uniform partitions of interval $[0,T]$ with $t^n_k=\frac{kT}{n}$ for all $n\in\mathbb{N}$ and all $k\in\{0,\ldots,n\}$, and let $X$ be some real-valued stochastic process defined on the interval $[0,T]$.

\begin{dfn} The normalized first- and second-order quadratic variations of $X$
taking along the partitions $(\pi_n)_{n\in\mathbb{N}}$ and corresponding to the value   $1/2<H<1$ are defined
as
\[
V_n^{(1)}(X,2)=n^{2H-1}\sum_{k=1}^n\big(\D^{(1)}_{k,n} X\big)^2,\quad
\D^{(1)}_{k,n} X=X(t^n_{k})-X(t^n_{k-1}),
\]
and
\[
V^{(2)}_n(X,2)=n^{2H-1}\sum_{k=1}^{n-1}\big(\D^{(2)}_{k,n} X\big)^2,\quad
\D^{(2)}_{k,n} X=X\big(t^n_{k+1}\big)-2X\big(t^n_{k}\big) +X\big(t^n_{k-1}\big).
\]
\end{dfn}
For simplicity, we shall omit  index $n$ for  points $t_k^n$ of partitions $\pi_n$.

It is known (see, e.g., Gladyshev \cite{glad}) that $V_n^{(1)}(B^H,2)\to T$ a.s. as $n\to\infty$. Also, it was proved in Benasi et al. \cite{bcij} and  Istas et al. \cite{IL} that $V_n^{(2)}(B^H,2)\to (4-2^{2H})T$ a.s. as $n\to\infty$.

Denote
\begin{align*}
V^{(1)}_n(B^H,2)_t=&n^{2H-1}\sum_{k=1}^{r(t)}(\D^{(1)}_{k,n}  B^H)^2\,,\\ V^{(2)}_n(B^H,2)_t=&n^{2H-1}\sum_{k=1}^{r(t)-1}(\D^{(2)}_{k,n}  B^H)^2,\quad t\in[0,T],
\end{align*}
where $r(t)=\max\{k\colon\ t_k\leq t\}=[\frac{tn}{T}]$. It is evident that
\begin{equation}\label{lyg}
\E V^{(1)}_n(B^H,2)_t=\rho(t)\quad\mbox{and}\quad \E V^{(2)}_n(B^H,2)_t=(4-2^{2H})\rho(t),
\end{equation}
where  $\rho(t)=\max\{t_k\colon\ t_k\ls t\}$.

The following  classical result will be used in the proof of Theorem \ref{konv1.}.
\begin{lem}\label{LevyOct}{(L\'evy-Octaviani inequality)}
Let $X_1,\ldots,X_n$ be independent random variables. Then  for any fixed $t,s\gs 0$
\[
\pr\bigg(\max_{1\ls i\ls n}\bigg\vert\sum_{j=1}^i X_j\bigg\vert>t+s\bigg)\ls \frac{\pr\big(\big\vert\sum_{j=1}^n X_j\big\vert>t\big)}{1-\max_{1\ls i\ls n} \pr\big(\big\vert\sum_{j=i}^{n} X_j\big\vert>s\big)}\,.\qquad\Box
\]
\end{lem}

\begin{thm}\label{konv1.} The following asymptotic property holds for the first- and second-order quadratic variations of fractional Brownian motion:
\begin{equation}\label{mars4}
\sup_{t\ls T}\big\vert V_n^{(i)}(B^H,2)_t-\E V_n^{(i)}(B^H,2)_t\big\vert=\mathcal{O}\big(n^{-1/2}\ln^{1/2} n\big)\quad\mbox{a.s.},\; i=1,2.
\end{equation}

\end{thm}
\begin{rem}\label{rem7} It follows from \eqref{lyg} and \eqref{mars4} that $\sup_{t\ls T}\big\vert V_n^{(i)}(B^H,2)_t\big\vert=\mathcal{O}(1)$ a.s., $i=1,2.$.
\end{rem}
\begin{proof}
We can consider
$V_n^{(i)}\big(B^H,2\big)$ as the square of the Euclidean norm of the  $n$-dimensional Gaussian vector $X_n$ with the components
\[
n^{H-1/2}\Delta^{(i)}_{k,n} B^H, \qquad 1\ls k\ls n-(i-1).
\]
Obviously, one can get a new $n$-dimensional Gaussian vector $\widetilde{X}_n$ with independent components applying the  linear transformation to $X_n$. It means that there exist  nonnegative real numbers $(\lambda^{(i)}_{1,n}, \ldots, \lambda^{(i)}_{n-(i-1),n})$ and such $n-(i-1)$-dimensional Gaussian vector $Y_n$ with independent Gaussian $\mathcal{N}(0,1)$-components that

\[
V_n^{(i)}\big(B^H,2\big)=\sum_{j=1}^{n-(i-1)}\lambda_{j,n}\big(Y^{(j)}_n\big)^2.
\]
The numbers $(\lambda^{(i)}_{1,n}, \ldots, \lambda^{(i)}_{n-(i-1),n})$ are the eigenvalues of the symmetric $n-(i-1)\times n-(i-1)$-matrix
\[
\Big(n^{2H-1}\E\big[\Delta^{(i)}_{j,n} B^H\Delta^{(i)}_{k,n} B^H\big]\Big)_{1\ls j,k\ls n-(i-1)}.
\]

Now we can apply the Hanson and Wright's inequality (see Hanson et al. \cite{HW} or Begyn \cite{begyn1}), and it  yields that for $\eps>0$
\begin{align}\label{nelyg4}
\pr\bigg(\bigg\vert\sum_{j=k}^{n-(i-1)}\lambda^{(i)}_{j,n}\big[\big(Y^{(j)}_n\big)^2-1\big]\bigg\vert\geq \eps\bigg) \ls& 2\exp\bigg({-}\min\bigg[\frac{C_1 \eps}{\l_{k,n}^{*(i)}}\,,\frac{C_2 \eps^2}{\sum_{j=k}^{n-(i-1)}(\l^{(i)}_{j,n})^2}\bigg]\bigg)\nonumber\\
\ls& 2\exp\bigg({-}\min\bigg[\frac{C_1 \eps}{\l_n^{*(i)}}\,,\frac{C_2 \eps^2}{\sum_{j=1}^{n-(i-1)}(\l^{(i)}_{j,n})^2}\bigg]\bigg),
\end{align}
where $C_1$, $C_2$ are nonnegative constants, $\l_{k,n}^{*(i)}=\max_{k\ls j\ls n-(i-1)}\l^{(i)}_{j,n}$, $\l_n^{*(i)}=\max_{1\ls j\ls n-(i-1)}\l^{(i)}_{j,n}$.

The evident equality holds:
\[
\sum_{j=1}^{n-(i-1)}\l^{(i)}_{j,n}=\E V_n^{(i)}\big(B^H,2\big).
\]
Furthermore, it follows from   (\ref{lyg}) that the sequence $\E V_n^{(i)}\big(B^H,2\big),\;n\geq 1$ is bounded. So, the sums $\sum_{j=1}^{n-(i-1)}\l^{(i)}_{j,n}$ are bounded as well.
It is easy to check that
\[
\sum_{j=1}^{n-(i-1)}(\l^{(i)}_{j,n})^2\ls \l_n^{*(i)}\sum_{j=1}^{n-(i-1)}\l^{(i)}_{j,n}.
\]
Therefore for any $0<\eps\ls 1$ the inequality  (\ref{nelyg4}) can be rewritten as
\begin{equation}\label{ivert}
\pr\bigg(\bigg\vert\sum_{j=i}^{n-(i-1)}\lambda_{j,n}\big[\big(Y^{(j)}_n\big)^2-1\big]\bigg\vert\geq \eps\bigg)
\ls 2\exp\bigg({-}\frac{K\eps^2}{\l_n^{*(i)}}\bigg),
\end{equation}
where $K$ is a positive constant.

Now we use L\'evy-Octaviani inequality (see Lemma \ref{LevyOct}) and evident inequality
\[
\frac{x}{1-x}\ls 2x \qquad \mbox{for}\quad 0< x\ls 1/2
\]
to obtain the bound
\begin{gather*}
\pr\bigg(\max_{1\ls k\ls n-(i-1)}\bigg\vert\sum_{j=1}^k\lambda^{(i)}_{j,n}\big[\big(Y^{(j)}_n\big)^2-1\big]\bigg\vert> 2\eps\bigg)\\
\ls \frac{2\exp\Big({-}\frac{K \eps^2}{\l_n^{*(i)}}\Big)}{1- 2\exp\Big({-}\frac{K \eps^2}{\l_n^{*(i)}}\Big)}\ls 4\exp\bigg({-}\frac{K \eps^2}{\l_n^{*(i)}}\bigg),
\end{gather*}
assuming that
\[
\exp\bigg({-}\frac{K \eps^2}{\l_n^{*(i)}}\bigg)\ls 1/4\quad\mbox{and}\quad 0<\varepsilon\leq 1.
\]
So, for the values of parameters mentioned above,
\[
\pr\Bigg(n^{2H-1}\max_{1\ls k\ls n-(i-1)}\bigg\vert  \sum_{j=1}^k  (\D^{(i)}_{j,n} B^H)^2-\sum_{j=1}^k  \E(\D^{(i)}_{j,n} B^H )^2 \bigg \vert> 2\eps\Bigg)\ls 4\exp\bigg({-}\frac{K \eps^2}{\l_n^{*(i)}}\bigg).
\]

Furthermore,
\[
\l_n^{*(i)}\ls K n^{2H-1}\max_{1\ls k\ls n-(i-1)}\sum_{j=1}^{n-(i-1)} \vert d_{jkn}^{(i)}\vert\,,
\]
where $d^{(i)}_{jkn}=\E\D^{(i)}_{j,n}B^H\D^{(i)}_{k,n}B^H$. From Gladyshev \cite{glad} and Begyn \cite{begyn1} we get
\begin{equation}\label{ivert1}
\l_n^{*(i)}\ls C n^{-1}.
\end{equation}
Now we set
\[
\eps_n^2=\frac{2C}{K}\,n^{-1}\ln n
\]
and conclude that
\begin{gather*}
\pr\Bigg(n^{2H-1}\max_{1\ls k\ls n-(i-1)}\bigg\vert  \sum_{j=1}^k  (\D^{(i)}_{j,n} B^H)^2-\sum_{j=1}^k  \E(\D^{(i)}_{k,n} B^H )^2 \bigg \vert> 2\eps_n\Bigg)\\
\ls 4\exp\bigg({-}2\ln n\bigg)=\frac{4}{n^2}\,.
\end{gather*}
It means that
\[
\sum_{n=0}^\infty \pr\Bigg(n^{2H-1}\max_{1\ls k\ls n-(i-1)}\bigg\vert  \sum_{j=1}^k  (\D^{(i)}_{j,n} B^H)^2-\sum_{j=1}^k  \E(\D^{(i)}_{j,n} B^H)^2 \bigg \vert> 2\eps_n\Bigg)<\infty.
\]
Finally, we get the  statement of the present theorem from the Borel-Cantelli lemma and the evident equality
\begin{align*}
&\sup_{t\ls T}\big\vert V_n^{(i)}(B^H,2)_t-\E V_n^{(i)}(B^H,2)_t\big\vert\\
&\quad=n^{2H-1}\max_{1\ls k\ls n-(i-1)}\bigg\vert  \sum_{j=1}^k  (\D^{(i)}_{j,n}  B^H)^2-\sum_{j=1}^k  \E(\D^{(i)}_{j,n}  B^H)^2 \bigg \vert\,.
\end{align*}

\end{proof}

\section{The rate of convergence of the first- and second-order quadratic variations of the solution of stochastic differential equation}

First, we formulate the following result from \cite{kubmel3} about convergence of first- and second-order quadratic variation.

\begin{thm}\label{var1}
Consider stochastic differential equation  ~\eqref{SIE},  where  function  $f$ is  Lipschitz  and  $g\in\mathcal{C}^{1+\a}$ for some  $0<\a<1$. Let $X$ be its solution.  Then
\begin{equation}\label{rezult}
V_n^{(i)}(X,2)\to c^{(i)}\int_0^T g^2\big(X(t)\big)\,d t\,\
\text{a.s.} \quad\text{as}\ n\to \infty,
\end{equation}
where
\[
c^{(i)}=\begin{cases}1& \mbox{for $i=1$},\\
(4-2^{2H})& \mbox{for $i=2$}.
\end{cases}
\]
\end{thm}
Second, we prove the following auxiliary result.
\begin{lem}\label{disc} Let $X$ be a solution of stochastic differential equation  ~\eqref{SIE}. Define a step-wise process $X^{\pi}$ that is a discretization of process  $X$:
\[
X^{\pi}_t=\begin{cases}X(t_k)& \mbox{for}\quad t\in[t_k,t_{k+1}),\ k=0,1,\ldots,n-2,\\
X(t_{n-1})& \mbox{for}\quad t\in[t_{n-1},t_n].
\end{cases}
\]
Then for any $p>\frac1H$ we have that
\begin{gather*}
\sup_{t\ls T}\vert X^{\pi}_t-X_t\vert= \mathcal{O}\big(n^{-1/p}\big).
\end{gather*}
\end{lem}

\begin{proof} Consider  $t\in [t_k,t_{k+1})$. We get immediately from the Love-Young inequality (\ref{1.2}) that
\begin{align*}
\vert X^{\pi}_t-X_t\vert=&\bigg\vert\int_{\rho^n(t)}^t f(X_s)\,ds+\int_{\rho^n(t)}^t g(X_s)\,dB^{H}_s\bigg\vert\\
\ls& T n^{-1}\sup_{t_k\ls t\ls t_{k+1}}\vert f(X_t)\vert
+C_{p,p} V_{p,\infty}(g(X);[t_k,t_{k+1}]) V_p(B^H;[t_k,t_{k+1}]).
\end{align*}
Further,
\begin{equation}\label{mars1}
\sup_{t_k\ls t\ls t_{k+1}}\vert f(X_t)\vert\ls \sup_{t\ls T} \vert f(X_t)-f(\xi)\vert+\vert f(\xi)\vert\ls LV_p(X;[0,t])+\vert f(\xi)\vert,
\end{equation}
where $L$ is a Lipschitz constant for $f$, and
\begin{align}\label{mars2}
V_{p,\infty}(g(X);[t_k,t_{k+1}])\ls& \vert g^\prime\vert_\infty V_p(X;[t_k,t_{k+1}])+\sup_{t\ls T} \vert g(X_t)-g(\xi)\vert+\vert g(\xi)\vert\nonumber\\
\ls& 2\vert g^\prime\vert_\infty V_p(X;[0,T])+\vert g(\xi)\vert.
\end{align}
We get the statement of the lemma from \eqref{mars1}, \eqref{mars2} and inequality (\ref{1.6}) .\end{proof}

Now we prove the main result of this section which specifies the rate of convergence in Theorem \ref{var1}.

\begin{thm}\label{var2} Let the conditions of Theorem \ref{var1} hold and, in addition, $\alpha>\frac1H-1$. Then
\begin{equation}\label{mars3}
V_n^{(i)}(X,2)-c^{(i)}\int_0^T g^2\big(X(t)\big)\,d t=\mathcal{O}\big(n^{-1/4}\ln^{1/4} n\big).
\end{equation}
\end{thm}

\begin{proof} Decompose the left-hand side of \eqref{mars3} into three parts:
\begin{align*}
 I_n^{(i)}:=&V_n^{(i)}(X,2)-c^{(i)}\int_0^T g^2\big(X(t)\big)\,d t = I_n^{(1,i)}+I_n^{(2,i)}+I_n^{(3,i)},
 \end{align*}
 where
\begin{align*} I_n^{(1,i)}=&n^{2H-1}\sum_{k=1}^{n-(i-1)} \big(\D_{k,n}^{(i)}X\big)^2-\sum_{k=1}^{n-(i-1)} g^2(X_{k-1+(i-1)})\big(\D_{k,n}^{(i)}B^H\big)^2,\\
I_n^{(2,i)}=&n^{2H-1}\sum_{k=1}^{n-(i-1)} g^2(X_{k-1+(i-1)})\big(\D_{k,n}^{(i)}B^H\big)^2\nonumber\\
&- \sum_{k=1}^{n-(i-1)} g^2(X_{k-1+(i-1)})\E\big(\D_{k,n}^{(i)} B^H\big)^2,\\
I_n^{(3,i)}=& n^{2H-1} \sum_{k=1}^{n-(i-1)} g^2(X_{k-1+(i-1)})\E\big(\D_{k,n}^{(i)}B^H\big)^2- c^{(i)}\int_0^T g^2(X_s)ds,
\end{align*}
and  $X_k=X(t_k)$. We start with the most simple term $I_n^{(3,i)}$ and get immediately, similarly to bounds contained in \eqref{1.3c},  that for any $p>\frac1H$
\begin{align*}
\vert I_n^{(3,i)}\vert\ls& c^{(i)}\sum_{k=1-(i-1)}^{n-(i-1)}\int^{t_{k+(i-1)}}_{t_{k-1+(i-1)}} \big\vert g^2(X(t_{k-1+(i-1)}))-g^2(X_s) \big\vert\,ds\\
\ls& 2c^{(i)}T \vert g^\prime\vert_\infty \sup_{t\ls T}\vert X^{\pi}_t-X_t\vert\cdot \big[\vert g^\prime\vert_\infty V_p(X;[0,T])+\vert g(\xi)\vert\big]\,.
\end{align*}

In order to estimate $I_n^{(2,i)}$, denote
\[
S^{(i)}_t=n^{2H-1}\sum_{k=1}^{r(t)-({i-1})}\big(\Delta^{(i)}_{k,n} B^H\big)^2\,,\quad t\in[0,T],\quad i=1,2.
\]

Then
\begin{eqnarray*}
n^{2H-1}\sum_{k=1}^{n-({i-1})}g^2(X_{k-1+(i-1)})\big(\Delta^{(i)}_{k,n} B^H\big)^2=\int_0^T g^2(X_t)\,dS^{(i)}_t
\end{eqnarray*}
and
\begin{align*}
&n^{2H-1}\sum_{k=1}^{n-({i-1})}g^2(X_{k-1+(i-1)})\big[\big(\Delta^{(i)}_{k,n} B^H\big)^2-\E \big(\Delta^{(i)}_{k,n} B^H\big)^2\big]\\
&\quad= \int_0^T g^2(X_t)\,d\big[S^{(i)}_t-\E S^{(i)}_t\big].
\end{align*}

Note that $1/p+1/2>1$ for $\frac1H<p<2$. Therefore, we obtain from the  Love-Young inequality and from (\ref{1.4})-(\ref{1.4a}) that
\begin{align*}
|I^{(2,i)}_n|=&\bigg\vert \int_0^T g^2(X_t)\,d\big[S^{(i)}_t-\E S^{(i)}_t\big]\bigg\vert\\
\ls& C_{p,2}V_{p,\infty}\big(g^2(X) ;[0,T]\big) V_2\big(S^{(i)}-\E S^{(i)} ;[0,T]\big)\\
\ls& C_{p,2}\big\{\osc\big( S^{(i)}-\E S^{(i)} ;[0,T]\big)\big\}^{1/2} V_{p,\infty}\big(g^2(X);[0,T]\big)\\
&\times V_1^{1/2}\big(S^{(i)}-\E S^{(i)} ;[0,T]\big)\\
\ls& 2C_{p,2}\Big(\sup_{t\ls T}\big\vert S^{(i)}_t-\E S^{(i)}_t\big\vert\Big)^{1/2} V^2_{p,\infty}\big(g(X);[0,T]\big)\\
&\times \bigg[n^{2H-1}\sum_{k=1}^{n-(i-1)}\big(\Delta^{(i)}_{k,n} B^H\big)^2+c^{(i)}T\bigg]^{1/2}.
\end{align*}

It follows from Theorem \ref{konv1.}, Remark \ref{rem7}, and (\ref{1.3c}) that the rate of convergence of $I^{(2,i)}_n$ is $O(n^{-1/4}\ln^{1/4} n)$.

It still remains to estimate $I^{(1,i)}_n$. Consider only  $i=2$,  the proof for $i=1$ is similar.

Denote
\begin{align*}
J_{k}^1=&\int_{t_k}^{t_{k+1}}[f(X_s)-f(X_k)]\,ds-\int_{t_{k-1}}^{t_{k}}[f(X_s)-f(X_k)]\,ds,\\
J_{k}^2=&\int_{t_{k-1}}^{t_{k}} \bigg(g(X_k)- g(X_s)
-\int_{s}^{t_k} g^\prime(X_k)f(X_k)\,du-\int_{s}^{t_k} g^\prime(X_k)g(X_k)\,dB^H_u\bigg)dB^H_s,\\
J_{k}^3=&\int_{t_k}^{t_{k+1}} \bigg( g(X_s)-g(X_k)
-\int_{t_k}^s g^\prime(X_k)f(X_k)\,du-\int_{t_k}^s g^\prime(X_k)g(X_k)\,dB^H_u\bigg)dB^H_s, \\
J_{k}^4=&g^\prime(X_k)f(X_k)\Big(\int_{t_k}^{t_{k+1}}(s-t_k)\,dB^H_s+\int_{t_{k-1}}^{t_{k}}(t_k-s)\,dB^H_s\Big),\\
J_{k}^5=&\frac{1}{2}\,g^\prime(X_k)g(X_k)\Big( \big(\Delta_{k,n}^{(1)} B^H\big)^2+\big(\Delta_{k+1,n}^{(1)} B^H\big)^2\Big),\qquad
J_{k}^6=g(X_k)\Delta^{(2)}_{k,n} B^H.
\end{align*}
Equalities
\begin{align*}
\int_{t_{k-1}}^{t_{k}} \bigg(\int_{s}^{t_k} g^\prime(X_k)g(X_k)\,dB^H_u\bigg)dB^H_s =& \frac{1}{2}\,g^\prime(X_k)g(X_k) \big(\Delta_k^{(1)} B^H\big)^2,\\
\int_{t_k}^{t_{k+1}}\bigg(\int_{t_k}^s g^\prime(X_k)g(X_k)\,dB^H_u\bigg)dB^H_s =&\frac{1}{2}\,g^\prime(X_k)g(X_k) \big(\Delta_{k+1,n}^{(1)} B^H\big)^2.
\end{align*}
(see Proposition \ref{2.1}) imply
\[
\Delta^{(2)}_{k,n} X=\sum_{l=1}^{6}J_{k}^l.
\]
Taking into account   Lipschitz property of $f$  and   Lemma \ref{disc}, we can  conclude that
\[
\vert f(X_t)-f(X_k)\vert\ls L \sup_{t\ls T}\vert X^{\pi}_t-X_t\vert=\mathcal{O}\big(n^{-1/p}\big).
\]
Therefore
\begin{align*}
\sum_{k=1}^{n-1}(J_{k}^1)^2\ls& 2Tn^{-1} \sum_{k=1}^{n-1}\int_{t_k}^{t_{k+1}} [f(X_s)-f(X_k)]^2\,ds\\
&+2Tn^{-1} \sum_{k=1}^{n-1}\int_{t_{k-1}}^{t_{k}} [f(X_s)-f(X_k)]^2\,ds\\
\ls& 4Tn^{-1} L^2\Big(\sup_{t\ls T}\vert X^{\pi}_t-X_t\vert\Big)^2  \\
=&\mathcal{O}\big(n^{-1-2/p}\big).
\end{align*}

Consider $J_{k}^2$. It follows from  equality (\ref{2.3}) that for any fixed  $t\in[t_{k-1},t_k]$
\begin{equation}\label{Ito1}
g(X_{k})-g(X_t)=\int^{t_k}_t g^\prime(X_s)f(X_s)\,ds+\int^{t_k}_t g^\prime(X_s)g(X_s)\,dB^H_s.
\end{equation}

Substituting equality (\ref{Ito1}) into $J_{k}^2$ we get
\begin{align}\label{mars5}
\vert J_{k}^2\vert\ls&\bigg\vert \int_{t_{k-1}}^{t_k}\int^{t_k}_s \big[ g^\prime(X_u)f(X_u)
- g^\prime(X_k)f(X_k)\big]du \,dB^H_s\bigg\vert\nonumber\\
&+\bigg\vert \int_{t_{k-1}}^{t_k}\int^{t_k}_s \big[ g^\prime(X_u)g(X_u)- g^\prime(X_k)g(X_k)\big]dB^H_u \,dB^H_s\bigg\vert\,.
\end{align}

Transforming identically  the  first term in the right-hand side of \eqref{mars5} and applying  to it   Love-Young inequality (\ref{1.2}), we conclude that for any $p>\frac1H$
\begin{align}\label{vert1}
&\bigg\vert\int_{t_{k-1}}^{t_k}\int^{t_k}_s \big[ g^\prime(X_u)f(X_u)- g^\prime(X_k)f(X_k)\big]du \,dB^H_s\bigg\vert \nonumber\\
&\quad\ls C_{p,1} V_1\bigg(\int_\cdot^{t_{k}} \big[ g^\prime(X_u)f(X_u)
- g^\prime(X_k)f(X_k)\big]du;[t_{k-1},t_k]\bigg)V_p\big( B^H;[t_{k-1},t_k]\big)\nonumber\\
&\quad\ls C_{p,1}V_p\big( B^H;[t_{k-1},t_k]\big)\int^{t_k}_{t_{k-1}} \big\vert g^\prime(X_u)f(X_u)- g^\prime(X_k)f(X_k)\big\vert\,du.
\end{align}
Henceforth  we consider the following interval of the values of $p$: $\frac1H<p< 1+\a$. Then it follows  from inequality (\ref{1.3a}) that the second term in the right-hand side of \eqref{mars5} admits the  bound:
\begin{align}\label{mars6}
&\bigg\vert \int_{t_{k-1}}^{t_k}\int^{t_k}_s \big[ g^\prime(X_u)g(X_u)- g^\prime(X_k)g(X_k)\big]dB^H_u \,dB^H_s\bigg\vert\nonumber\\
&\quad\ls C_{p,p/\alpha} V_{p/\alpha}\bigg(\int^{t_{k}}_\cdot \big[ g^\prime(X_u)g(X_u)
- g^\prime(X_k)g(X_k)\big]dB^H_u;[t_{k-1},t_k]\bigg)\nonumber\\
&\qquad\times V_p\big( B^H;[t_{k-1},t_k]\big)\nonumber\\
&\quad\ls 2C^2_{p,p/\alpha}V_{p/\alpha}\big( g^\prime(X_u)g(X_u)
;[t_{k-1},t_k]\big) V^2_p\big( B^H;[t_{k-1},t_k]\big).
\end{align}

We conclude  from \eqref{mars5}--\eqref{mars6} that
\begin{align*}
\vert  J^2_k\vert
\ls& Tn^{-1} C_{p,1} V_{p/\alpha}\big( g^\prime(X)f(X);[t_k,t_{k+1}]\big)V_p\big( B^H;[t_k,t_{k+1}]\big)\\
&+ 2C^2_{p,p/\alpha} V_{p/\alpha}\big( g^\prime(X)g(X);[t_k,t_{k+1}]\big)V^2_p\big( B^H;[t_k,t_{k+1}]\big).
\end{align*}
Applying  inequalities (\ref{1.4a}) and (\ref{1.5}) we immediately obtain  that
\begin{align*}
\sum_{k=1}^n \big( J^2_k\big)^2\ls& 2T^2 C^2_{p,1} n^{-2}\max_{0\ls k\ls n-1}\big[ V_{p/\alpha}\big( g^\prime(X)f(X);[t_k,t_{k+1}]\big)
V_p\big( B^H;[t_k,t_{k+1}]\big)\big]\\
&\times V_{p/\alpha}\big( g^\prime(X)f(X);[0,T]\big)V_p\big( B^H;[0,T]\big)\\
&+ 4C^4_{p,p/\alpha}\max_{0\ls k\ls n-1}\big[ V_{p/\alpha}\big( g^\prime(X)g(X);[t_k,t_{k+1}]\big)V^3_p\big( B^H;[t_k,t_{k+1}]\big)\big]\\
&\times V_{p/\alpha}\big( g^\prime(X)g(X);[0,T]\big)V_p\big( B^H;[0,T]\big)\\
\ls& 2 T^2 C^2_{p,1} n^{-2}\max_{0\ls k\ls n-1}\big[V_p\big( B^H;[t_k,t_{k+1}]\big)\big]
V^2_{p/\alpha,\infty}\big( g^\prime(X);[0,T]\big)\\
&\times V^2_{p,\infty}\big( f(X);[0,T]\big)V_p\big( B^H;[0,T]\big)\\
&+ 4C^4_{p,p/\alpha}\max_{0\ls k\ls n-1}\big[V^3_p\big( B^H;[t_k,t_{k+1}]\big)\big]\\
&\times V^2_{p/\alpha,\infty}\big( g^\prime(X);[0,T]\big) V^2_{p/\alpha,\infty}\big( g(X);[0,T]\big)V_p\big( B^H;[0,T]\big).
\end{align*}
It follows from the inequalities (\ref{1.3b})--(\ref{1.3d}) that the values of the variations
\[
V_{p,\infty}\big( f(X);[0,T]\big), \qquad V_{p/\alpha,\infty}\big( g(X);[0,T]\big), \qquad\text{and}\qquad V_{p/\alpha,\infty}\big( g^\prime(X);[0,T]\big)
\]
are finite. Therefore  we get from (\ref{1.6}) that
\[
\sum_{k=1}^n \big( J^2_k\big)^2= \mathcal{O}(n^{-2-1/p})+  \mathcal{O}(n^{-3/p})=\mathcal{O}(n^{-3/p}).
\]

The similar reasonings lead to the similar bound for $ J^3_k$, and we conclude that
\[
\sum_{k=0}^{n-1} [J^2_k+J^3_k]^2=\mathcal{O}\big(n^{-3/p}\big).
\]

Consider $J_k^4$. It consists of two terms that can be estimated in a similar way. Applying inequalities \eqref{1.2}  and (\ref{1.6}),  we obtain the following bound for the first term:
\begin{align*}
& \sum_{k=1}^{n-1} \big[g^\prime(X_k)f(X_k)\big]^2\bigg(\int_{t_k}^{t_{k+1}}(s-t_k)\,dB^H_s\bigg)^2\\
&\quad\ls C^2_{p,1}\sum_{k=1}^{n-1} \big[g^\prime(X_k)f(X_k)\big]^2 (t_{k+1}-t_k)^2V^2_p\big( B^H;[t_k,t_{k+1}]\big)\\
&\quad\ls n^{-1}T^2C^2_{p,1}\max_{1\ls k\ls n-1}\big[g^\prime(X_k)f(X_k) V_p\big( B^H;[t_k,t_{k+1}]\big)\big]^2
=\mathcal{O}\big(n^{-1-2/p}\big).
\end{align*}
As a consequence,
\[
\sum_{k=1}^{n-1} \big(J^4_k\big)^2=\mathcal{O}\big(n^{-1-2/p}\big).
\]

Furthermore, note that under our assumptions $\sup_{s\in[0,T]}|g(X_s)|<\infty$ and $\sup_{s\in[0,T]}|g^\prime(X_s)|<\infty$ a.s. Therefore we have for the first term in $J_k^5$ that
\begin{align*}
\sum_{k=1}^{n-1}\Big[ g^\prime(X_k)g(X_k)
\big(\Delta_{k+1,n}^{(1)} B^H\big)^2\big]^2
\ls& \max_{1\ls k\ls n-1}\big[g^\prime(X_k)g(X_k)\big]^2
\sum_{k=0}^{n-1}\big(\Delta_{k+1,n}^{(1)} B^H\big)^4\\
=&\mathcal{O}\big(n^{1-4/p}\big).
\end{align*}
The second term is bounded in a similar way, and we conclude that
\[
\sum_{k=1}^{n-1} \big(J^5_k\big)^2=\mathcal{O}\big(n^{1-4/p}\big).
\]

Thus
\[
\bigg(\sum_{l=1}^{5}J_{k}^l\bigg)^2=\mathcal{O}\big(n^{-1-2/p}\lor n^{-3/p}\lor n^{1-4/p}\big)=\mathcal{O}\big(n^{1-4/p}\big).
\]
At last,
\begin{align}\label{konv}
n^{2H-1}\sum_{k=1}^{n-1}\big[\Delta^{(2)}_{k,n}
X-g(X_k)\Delta^{(2)}_{k,n} B^H\big]^2=&\mathcal{O}\big(n^{1-4/p+2H-1}\big)\nonumber\\
=&\mathcal{O}\big(n^{-4/p+2H}\big)
\end{align}
for any $\frac1H< p< 1+\a$.  Set  $1/p=H-\eps$  for $\eps<(H/2-1/16)\wedge(H-\frac{1}{1+\alpha})$. Then
\begin{align}\label{konv1}
&V_n^{(2)}(X,2)-c^{(2)}\int_0^T g^2(X_s)\,ds\nonumber\\
&\quad=\mathcal{O}\big(n^{-4/p+2H}\big)+\mathcal{O}\big(n^{-1/4}\ln^{1/4} n\big)+\mathcal{O}\big(n^{-1/p}\big)\nonumber\\
&\quad=\mathcal{O}\big(n^{-2H+4\eps}\big)+\mathcal{O}\big(n^{-1/4}\ln^{1/4} n\big)+\mathcal{O}\big(n^{-H+\eps}\big)\nonumber\\
&\quad=\mathcal{O}\big(n^{-1/4}\ln^{1/4} n\big).
\end{align}
\end{proof}

\section{The rate of convergence of estimators of Hurst index}

Consider the following statistics: $$R_n^{(i)}=\frac{\sum_{k=1}^{2n-(i-1)}(\Delta^{(i)}_{k,2n}X)^2}{\sum_{k=1}^{n-(i-1)}(\Delta^{(i)}_{k,n}X)^2}$$ and construct the following estimate of Hurst index $H$:
\[
\widehat{H}^{(i)}_n=\bigg(\frac{1}{2} - \frac{1}{2\ln2}\ln R_n^{(i)}\bigg){\bf 1}_{\widetilde{C}_n},
\]
where
\[
\widetilde{C}_n=\bigg\{2^{-1}\big(1-2 n^{-1/4}(\ln n)^{1/4+\b}\big)\ls R_n^{(i)}\ls 1+2 n^{-1/4}(\ln n)^{1/4+\b} \bigg\},\qquad\beta>0.
\]
 Further, introduce the following notation: $g^{(i)}(T)=c^{(i)}\int_0^T g^2(X_s)ds$.
\begin{thm}\label{mainthm} Let conditions of Theorem \ref{var1} hold with $\alpha>\frac1H-1.$ Also, let $X$  be a solution of (\ref{SIE}) and assume that random variable $g^{(i)}(T)$ is separated from zero:  there exists a constant $c_0>0$ such that $g^{(i)}(T)\gs c_0$ a.s. Then $\widehat{H}^{(i)}_n$ is a strongly consistent estimator of the Hurst index $H$ and the following rate  of convergence holds\emph{:}
\[
\vert \widehat H^{(i)}_n-H\vert=\mathcal{O}\big(n^{-1/4}(\ln n)^{1/4+\b}\big)\quad\mbox{a.s.},
\]
for any $\b>0$.
\end{thm}
\begin{proof} Consider a sequence $1>\d_n\downarrow 0$ as $n\rightarrow\infty$. It will be specified later on.  Introduce the events
\[
C_n=\bigg\{\frac12(1-\d_n)\ls R_n^{(i)}\ls 1+\d_n \bigg\}.
\]
Also, introduce the notations
\[
A^{(i)}_n=V_{2n}^{(i)}(X,2)\quad \mbox{and}\quad B^{(i)}_n=V_n^{(i)}(X,2)
\]
and note that $2^{2H-1}R_n^{(i)}=\frac{A_n^{(i)}}{B_n^{(i)}}.$
Then $$ C_n=\bigg\{2^{2H-2}(1-\d_n)\ls\frac{A_n^{(i)}}{B_n^{(i)}}\ls 2^{2H-1}(1+\d_n )\bigg\},$$
and   estimate $\widehat{H}^{(i)}_n$ has a form
\[
\widehat{H}^{(i)}_n=\bigg(\frac{1}{2} - \frac{1}{2\ln2}\ln R_n^{(i)}\bigg){\bf 1}_{{C}_n}.
\]
It is easy to see that $\overline C_n:=\Omega\backslash C_n$ has a form
\begin{align}
\overline C_n
=&\bigg\{\frac{A^{(i)}_n}{B^{(i)}_n}< 2^{2H-2}(1-\d_n)\bigg\}\bigcup \bigg\{\frac{A^{(i)}_n}{B^{(i)}_n}> 2^{2H-1}(1+\d_n)\bigg\}\nonumber\\
\subset& \bigg\{\bigg\vert\frac{A^{(i)}_n}{B^{(i)}_n}-1\bigg\vert>\d_n\bigg\}.
\end{align}

Then
\begin{align*}
\widehat{H}^{(i)}_n=&H{\bf 1}_{C_n}-\frac{1}{2\ln2}\ln \frac{(2n)^{2H-1}V_{2n}^{(i)}(X,2)}{ n^{2H-1}V_n^{(i)}(X,2)}\,{\bf 1}_{C_n}\\
=&H{\bf 1}_{C_n}-\frac{1}{2\ln2}\ln \frac{A^{(i)}_n}{B^{(i)}_n}\,{\bf 1}_{C_n}.
\end{align*}
The latter representation implies that
\begin{align}\label{1.7}
\big\vert\widehat{H}^{(i)}_n-H\big\vert\ls& H{\bf 1}_{\big\{\big\vert\frac{A^{(i)}_n}{B^{(i)}_n}-1\big\vert>\d_n\big\}}+\frac{1}{2\ln2}\bigg\vert\ln \frac{A^{(i)}_n}{ B^{(i)}_n}\bigg\vert\,{\bf 1}_{\big\{1-\d_n\ls\frac{A^{(i)}_n}{B^{(i)}_n}\ls 1+\d_n\big\}}\nonumber\\
&-\bigg(\frac{1}{2\ln2}\,\ln \frac{A^{(i)}_n}{ B^{(i)}_n}\bigg){\bf 1}_{\big\{2^{2H-2}(1-\d_n)\ls\frac{A^{(i)}_n}{B^{(i)}_n}< 1-\d_n \big\}}\nonumber\\
&+\bigg(\frac{1}{2\ln2}\,\ln \frac{A^{(i)}_n}{ B^{(i)}_n}\bigg){\bf 1}_{\big\{1+\d_n\ls\frac{A^{(i)}_n}{B^{(i)}_n}< 2^{2H-1}(1+\d_n ) \big\}}
:=\sum_{l=1}^4 L_n^l.
\end{align}
In what follows we need an elementary inequalities:  $-\ln(1-x)\leq 2\ln(1+x)\ls 2x$ provided that $0\ls x\ls 1/2$.

Consider $L_n^2$. We divide it in two parts. As to the first part, it is obvious that
\[
\bigg(\ln \frac{A^{(i)}_n}{ B^{(i)}_n}\bigg){\bf 1}_{\big\{1-\d_n\ls \frac{A^{(i)}_n}{B^{(i)}_n}<1\big\}}=\bigg(\ln\bigg[1-\bigg(1-\frac{A^{(i)}_n}{ B^{(i)}_n}\bigg)\bigg]\bigg){\bf 1}_{\big\{1-\d_n\ls \frac{A^{(i)}_n}{B^{(i)}_n}<1\big\}},
\]
and
\[
1-\d_n\ls \frac{A^{(i)}_n}{B^{(i)}_n}<1\quad\text{implies that}\quad 0<1- \frac{A^{(i)}_n}{ B^{(i)}_n}\ls \d_n.
\]
Applying inequality  $-\ln(1-x)\ls 2x$, $0\ls x\ls 1/2$, we deduce that for $\d_n\ls 1/2$
\[
\bigg(-\ln \frac{A^{(i)}_n}{ B^{(i)}_n}\bigg){\bf 1}_{\big\{1-\d_n\ls \frac{A^{(i)}_n}{B^{(i)}_n}<1\big\}}\ls 2\bigg(1-\frac{A^{(i)}_n}{ B^{(i)}_n}\bigg){\bf 1}_{\big\{1-\d_n\ls \frac{A^{(i)}_n}{B^{(i)}_n}<1\big\}}\ls 2\d_n{\bf 1}_{\big\{1-\d_n\ls \frac{A^{(i)}_n}{B^{(i)}_n}<1\big\}}.
\]
As to the second part,
\begin{align*}
\bigg(\ln \frac{A^{(i)}_n}{ B^{(i)}_n}\bigg){\bf 1}_{\big\{1\ls\frac{A^{(i)}_n}{B^{(i)}_n}\ls  1+\d_n\big\}}=&\bigg(\ln\bigg[1+\bigg(\frac{A^{(i)}_n}{ B^{(i)}_n}-1\bigg)\bigg]\bigg){\bf 1}_{\big\{1\ls\frac{A^{(i)}_n}{B^{(i)}_n}\ls 1+\d_n\big\}}\\
\ls&\bigg(\frac{A^{(i)}_n}{ B^{(i)}_n}-1\bigg){\bf 1}_{\big\{1\ls\frac{A^{(i)}_n}{B^{(i)}_n}\ls 1+\d_n\big\}}
\ls \d_n{\bf 1}_{\big\{1\ls\frac{A^{(i)}_n}{B^{(i)}_n}\ls 1+\d_n\big\}}.
\end{align*}

Consider $L_n^3$.  From here we easy deduce that
\begin{align*}&-\bigg(\frac{1}{2\ln2}\,\ln \frac{A^{(i)}_n}{ B^{(i)}_n}\bigg){\bf 1}_{\big\{2^{2H-2}(1-\d_n)\ls\frac{A^{(i)}_n}{B^{(i)}_n}< 1-\d_n \big\}}\\
&\quad\ls -\frac{1}{2\ln2}\big[\ln \big(2^{2H-2}(1-\d_n)\big)\big]{\bf 1}_{\big\{2^{2H-2}(1-\d_n)\ls\frac{A^{(i)}_n}{B^{(i)}_n}< 1-\d_n \big\}}\\
&\quad\ls \bigg((1-H)-\frac{\ln (1-\d_n)}{2\ln2}\bigg){\bf 1}_{\big\{2^{2H-2}(1-\d_n)\ls\frac{A^{(i)}_n}{B^{(i)}_n}< 1-\d_n \big\}}\\
&\quad\ls \bigg(1-H+\frac{\d_n}{\ln2}\bigg){\bf 1}_{\big\{2^{2H-2}(1-\d_n)\ls\frac{A^{(i)}_n}{B^{(i)}_n}< 1-\d_n\big\}}\\
&\quad\ls (1+2\d_n) {\bf 1}_{\big\{\frac{A^{(i)}_n}{B^{(i)}_n}< 1-\d_n \big\}}.
\end{align*}

The term $L_n^4$ is estimated similarly as the second part of $L_n^2$. Thus we get
\begin{align*}
\bigg(\ln \frac{A^{(i)}_n}{ B^{(i)}_n}\bigg){\bf 1}_{\big\{1+\d_n\ls\frac{A^{(i)}_n}{B^{(i)}_n}\ls  2^{2H-1}(1+\d_n)\big\}}\ls&
\bigg(\frac{A^{(i)}_n}{ B^{(i)}_n}-1\bigg){\bf 1}_{\big\{1+\d_n\ls\frac{A^{(i)}_n}{B^{(i)}_n}\ls  2^{2H-1}(1+\d_n)\big\}}\\
\ls& (1+2\d_n){\bf 1}_{\big\{\frac{A^{(i)}_n}{B^{(i)}_n}> 1+\d_n\big\}}.
\end{align*}
Summarizing, we conclude that
\begin{align*}
\vert \widehat{H}^{(i)}_n-H\vert\ls&(1+2\d_n){\bf 1}_{\big\{\big\vert\frac{A^{(i)}_n}{B^{(i)}_n}-1\big\vert>\d_n\big\}}+2\d_n{\bf 1}_{\big\{1-\d_n\ls \frac{A^{(i)}_n}{B^{(i)}_n}\ls 1+\d_n\big\}}\\
\ls&(1+2\d_n){\bf 1}_{\big\{\big\vert\frac{A^{(i)}_n}{B^{(i)}_n}-1\big\vert>\d_n\big\}}+2\d_n.
\end{align*}
Now, let $\b>0$. Note that
\begin{align*}
&\bigg\{\bigg\vert\frac{A^{(i)}_n}{B^{(i)}_n}-1\bigg\vert>\d_n\bigg\}\\
&\quad\subset\bigg\{\bigg\vert\frac{A^{(i)}_n}{B^{(i)}_n}-1\bigg\vert>\d_n, B^{(i)}_n\gs (\ln n)^{-\b} \bigg\}\bigcup\big\{B^{(i)}_n< (\ln n)^{-\b} \big\}\\
&\quad=\big\{\vert A^{(i)}_n-B^{(i)}_n\vert> \d_n B^{(i)}_n, B^{(i)}_n\gs (\ln n)^{-\b}\big\}\cup\big\{B^{(i)}_n< (\ln n)^{-\b} \big\}\\
&\quad\subset \big\{\big\vert A^{(i)}_n-B^{(i)}_n\big\vert> \d_n (\ln n)^{-\b}\big\}\cup\big\{B^{(i)}_n< (\ln n)^{-\b} \big\}.
\end{align*}
Therefore
\[
\vert \widehat{H}^{(i)}_n-H\vert\ls(1+2\d_n){\bf 1}_{\{\vert A^{(i)}_n-B^{(i)}_n\vert> \d_n (\ln n)^{-\b}\}\cup\{B^{(i)}_n< (\ln n)^{-\b}\}}+2\d_n.
\]

It follows from (\ref{konv1})  that
\[
\vert A^{(i)}_n-B^{(i)}_n\vert=\mathrm{O}\big(n^{-1/4}\ln^{1/4} n\big)
\]
and
\[
\bigg\vert B^{(i)}_n-c^{(i)}\int_0^T g^2(X_s)ds\bigg\vert=\mathrm{O}\big(n^{-1/4}\ln^{1/4} n\big).
\]
Obviously, for any $n>\exp\{\big(\frac{2}{c_0}\big)^{\frac{1}{\beta}}\}$ we have that $g^{(i)}(T)\gs c_0\geq 2(\ln n)^{-\b}$ a.s. and
\[
\big\{B^{(i)}_n< (\ln n)^{-\b} \big\}=\big\{B^{(i)}_n< (\ln n)^{-\b},g^{(i)}(T)\gs 2(\ln n)^{-\b} \big\}.
\]

Now, let $\d_n<(\ln n)^{-\b}$.  Then it is not hard to deduce that
\begin{align*}
&\big\{B^{(i)}_n< (\ln n)^{-\b},g^{(i)}(T)\gs 2(\ln n)^{-\b} \big\}\\
&\quad=\big\{B^{(i)}_n< (\ln n)^{-\b},g^{(i)}(T)\gs 2(\ln n)^{-\b},B^{(i)}_n<g^{(i)}(T)-\d_n \big\}\\
&\quad\subset\big\{\vert B^{(i)}_n-g^{(i)}(T)\vert>\d_n \big\}.
\end{align*}

Therefore,
\[
\big\{B^{(i)}_n< (\ln n)^{-\b} \big\}\subset\big\{\vert B^{(i)}_n-g^{(i)}(T)\vert>\d_n \big\}
\]
if $n>\exp\{\big(\frac{2}{c_0}\big)^{\frac{1}{\beta}}\}$.

Finally, specify $\d_n$. More precisely, set   $\d_n=n^{-1/4}(\ln n)^{1/4+2\b}$, $\b>0$. Note that $\d_n<(\ln n)^{-\b}$ for sufficiently large $n$ and, moreover,
\[
\frac{\mathcal{O}\big(n^{-1/4}\ln^{1/4} n\big)}{\d_n(\ln n)^{-\b}}=\frac{\mathcal{O}\big(n^{-1/4}\ln^{1/4} n\big)}{n^{-1/4}(\ln n)^{1/4+\b}} \longrightarrow 0\qquad\mbox{a.s.}\quad\mbox{as}\ n\to\infty.
\]
The latter relation together with Theorem \ref{var2} imply that for any $\omega\in\Omega^{\prime}$ with $P(\Omega^{\prime})=1$ there exists $n_0=n_0(\omega)$ such that for any $n>n_0$
\[
{\bf 1}_{\big\{\vert A^{(i)}_n-B^{(i)}_n\vert> \d_n (\ln n)^{-\b}\big\}\cup\big\{B^{(i)}_n< (\ln n)^{-\b}\big\}}=0\quad\mbox{a.s.},
\]
and we obtain the proof.

\end{proof}

\section{Simulation results}

Consider  fractional Ornstein-Uhlenbeck process that is the solution of the linear stochastic differential equation
\[
dX_t= - X_tdt + dB^H_t,\qquad X_0=0.
\] with the step $0.05$ and for increasing (in the logarithmic scale) number $n$
 of points  from $n=10^2$ to $n=10^6$.
Table 1 presents the values of the difference $\vert \widehat H^{(1)}_n-H\vert$ for the values of $H$ from $0.55$ to $0.95$.
We can conclude that the difference $\vert \widehat H^{(1)}_n-H\vert$ decreases rapidly in $n$ and for fixed value of $n$ increases in $H$.
Table 2 demonstrates that the rate of convergence  agrees with Theorem \ref{mainthm}, at least, for $\beta=0.05$.
 Moreover, we can see from Table 3 that in the case of the linear equation  the rate of convergence  for $H\in(0.5, 0.7)$  can be estimated by $n^{-1/2}(\ln n)^{1/2}$.
%

\newpage

\begin{table}[h]
\caption{  $\vert\widehat H^{(1)}_n-H\vert$ }\label{t1}
\centering
 \vspace{2mm}
 {\footnotesize{
\begin{tabular}{llllllllll}
\hline
&\multicolumn{9}{c}{$n$ points}
\\
\cline{2-10}
\noalign{\smallskip}
\textbf{H} &\multicolumn{1}{c}{$100$} &\multicolumn{1}{c}{$250$} &\multicolumn{1}{c}{$1000$} &\multicolumn{1}{c}{$2500$} &\multicolumn{1}{c}{$10^4$} &\multicolumn{1}{c}{$2.5\cdot10^4$} & \multicolumn{1}{c}{$10^5$} & \multicolumn{1}{c}{$2.5\cdot10^5$} & \multicolumn{1}{c}{$10^6$} \\
\hline \noalign{\smallskip} \textbf{0.55} & 0,08401& 0,05488& 0,02124& 0,01467& 0,00777& 0,00549&  0,00195& 0,00160& 0,00079\\
\textbf{0.6} &  0,07216& 0,04145& 0,02213& 0,01286& 0,00683& 0,00466& 0,00214& 0,00137& 0,00069\\
\textbf{0.65} & 0,07761& 0,04811& 0,01972& 0,01296& 0,00626& 0,00414& 0,00210& 0,00144&0,00066\\
\textbf{0.7}& 0,05364& 0,03403& 0,02023& 0,01219& 0,00608& 0,00341& 0,00183& 0,00125&0,00065\\
\textbf{0.75}   & 0,06485& 0,03798& 0,02187& 0,01147& 0,00707& 0,00424& 0,00211&0,00140& 0,00083\\
\textbf{0.8 }& 0,05938& 0,03884& 0,02040&  0,01307& 0,00791& 0,00528& 0,00303& 0,00227&0,00120\\
\textbf{0.85} & 0,04666& 0,03577& 0,02105& 0,01684&  0,01011& 0,00753& 0,00511& 0,00384&0,00249\\
\textbf{0.9} &  0,06311& 0,04642& 0,03037& 0,02338& 0,01667& 0,01352& 0,00984& 0,00801&0,00599\\
\textbf{0.95} & 0,06219& 0,04763& 0,03488& 0,02907& 0,02295& 0,02018& 0,01640& 0,01448&0,01213\\
\hline
\end{tabular}
}}
\end{table}

\begin{table}[h!]
\caption{  $\vert\widehat H^{(1)}_n-H\vert\cdot n^{0.25}(\ln n)^{-0.3}$ }\label{t1}
\centering
 \vspace{2mm}
 {\footnotesize{
\begin{tabular}{llllllllll}
\hline
&\multicolumn{9}{c}{$n$ points}
\\
\cline{2-10}
\noalign{\smallskip}
\textbf{H} &\multicolumn{1}{c}{$100$} &\multicolumn{1}{c}{$250$} &\multicolumn{1}{c}{$1000$} &\multicolumn{1}{c}{$2500$} &\multicolumn{1}{c}{$10^4$} &\multicolumn{1}{c}{$2.5\cdot10^4$} & \multicolumn{1}{c}{$10^5$} & \multicolumn{1}{c}{$2.5\cdot10^5$} & \multicolumn{1}{c}{$10^6$} \\
\hline
\noalign{\smallskip}
\textbf{0.55} & 0.16802&    0.13070&    0.06689&    0.05596& 0.03991& 0.03449&  0.01663&    0.01681&    0.01136 \\
\textbf{0.6} &  0.14433& 0.09873& 0.06969&  0.04906&    0.03506&    0.02923&    0.01828&    0.01435&    0.00994\\
\textbf{0.65} & 0.15522&    0.11458&    0.06212&    0.04942&    0.03215&    0.02602&    0.01793&    0.01510& 0.00954\\
\textbf{0.7}& 0.10727&  0.08104&    0.06370&    0.04652&    0.03125&    0.02142&    0.01565& 0.01311& 0.00931\\
\textbf{0.75}& 0.12970& 0.09044&    0.06888&    0.04376&    0.03631&    0.02664&    0.01804& 0.01471& 0.01189\\
\textbf{0.8 }& 0.11877& 0.09250&    0.06426&    0.04987&    0.04066&    0.03313&    0.02590& 0.02382& 0.01723\\
\textbf{0.85} & 0.09331&    0.08519&    0.06628&    0.06425&    0.05193&    0.04726& 0.04363& 0.04031&  0.03580\\
\textbf{0.9} &  0.12622& 0.11055&   0.09563&    0.08920&    0.08563&    0.08485&    0.08410&    0.08413&    0.08612\\
\textbf{0.95} & 0.12438&    0.11344&    0.10984&    0.11089&    0.11789&    0.12669& 0.14009&   0.15206&    0.17450\\
\hline
\end{tabular}
}}
\end{table}

\begin{table}[h!]
\caption{  $\vert\widehat H^{(1)}_n-H\vert\cdot n^{0.5}(\ln n)^{-0.5}$ }\label{t1}
\centering
 \vspace{2mm}
 {\footnotesize{
\begin{tabular}{llllllllll}
\hline
&\multicolumn{9}{c}{$n$ points}
\\
\cline{2-10}
\noalign{\smallskip}
\textbf{H} &\multicolumn{1}{c}{$100$} &\multicolumn{1}{c}{$250$} &\multicolumn{1}{c}{$1000$} &\multicolumn{1}{c}{$2500$} &\multicolumn{1}{c}{$10^4$} &\multicolumn{1}{c}{$2.5\cdot10^4$} & \multicolumn{1}{c}{$10^5$} & \multicolumn{1}{c}{$2.5\cdot10^5$} & \multicolumn{1}{c}{$10^6$} \\
\hline
\noalign{\smallskip}
\textbf{0.55} & 0.59403 &  0.56033 &  0.38778 &  0.39788 &  0.38848 &  0.41418 &  0.27528 &  0.34462 &  0.32254 \\
\textbf{0.6} & 0.51028 &  0.42327 &  0.40402 &  0.34882 &  0.34127 &  0.35107 &  0.30255 &  0.29419 &  0.28218 \\
\textbf{0.65} & 0.54879 &  0.49122 &  0.36011 &  0.35142 &  0.31289 &  0.31248 &  0.29686 &  0.30960 &  0.27064 \\
\textbf{0.7} & 0.37926 &  0.34744 &  0.36932 &  0.33076 &  0.30417 &  0.25725 &  0.25904 &  0.26879 &  0.26421 \\
\textbf{0.75} &  0.45857 &  0.38776 &  0.39935 &  0.31116 &  0.35343 &  0.31994 &  0.29864 &  0.30156 &  0.33742 \\
\textbf{0.8} &  0.41990 &  0.39658 &  0.37253 &  0.35462 &  0.39573 &  0.39783 &  0.42871 &  0.48825 &  0.48905 \\
\textbf{0.85} &  0.32990 &  0.36523 &  0.38428 &  0.45683 &  0.50541 &  0.56751 &  0.72214 &  0.82624 &  1.01618 \\
\textbf{0.9} &  0.44627 &  0.47396 &  0.55444 &  0.63424 &  0.83341 &  1.01898 &  1.39207 &  1.72453 &  2.44410 \\
\textbf{0.95} &  0.43976 &  0.48637 &  0.63677 &  0.78848 &  1.14744 &  1.52138 &  2.31880 &  3.11678 &  4.95257 \\
\hline
\end{tabular}
}}
\end{table}

\end{document}